\documentclass{article}

\usepackage[left=30truemm, right=30truemm]{geometry}
\usepackage{amsmath,amssymb,amsthm}
\usepackage{bm}
\usepackage{mathtools}
\usepackage{graphicx}
\usepackage{enumitem}[shortlabels]

\usepackage{xcolor,graphicx}
\usepackage{tikz}
\usetikzlibrary{calc}
\usepackage{subcaption}

\usepackage[pdfauthor={derajan},pdfstartview=XYZ,bookmarks=true, colorlinks=true,linkcolor=blue,urlcolor=magenta,citecolor=blue,pdftex,bookmarks=true,linktocpage=true,hyperindex=true]{hyperref}

\theoremstyle{definition}
\newtheorem{theorem}{Theorem}[]

\newtheorem{corollary}[theorem]{Corollary}

\newtheorem{claim}[]{Claim}

\newtheorem{conjecture}[theorem]{Conjecture}

\numberwithin{equation}{section}

\title{Degree sum conditions and a 2-factor with a bounded number of cycles in claw-free graphs}
\author{Masaki Kashima\thanks{Faculty of Science and Technology, Keio University, Yokohama, Japan. email: masaki.kashima10@gmail.com}}

\begin{document}

\maketitle

\begin{abstract}
    A claw-free graph is a graph that does not contain $K_{1,3}$ as an induced subgraph, and a 2-factor is a 2-regular spanning subgraph of a graph.
    In 1997, Ryj\'{a}\v{c}ek introduced the closure concept of claw-free graphs, and Hamilton cycles and related structures in claw-free graphs have been intensively studied via the closure concept.
    In this paper, using the closure concept, we show that for a claw-free graph $G$ of order $n$, if every independent set $I$ of $G$ satisfies $|I|\leq \delta_G(I)-1$ and $G$ satisfies $\sigma_{k+1}(G)\geq n$, then $G$ has a 2-factor with at most $k$ cycles, where $\delta_G(I)$ denotes the minimum degree of the vertices in $I$.
    As a corollary of the result, we show that every claw-free graph $G$ with $\delta(G)\geq \alpha(G)+1$ has a 2-factor with at most $\alpha(G)$ cycles, which partially solves a conjecture by Faudree et al. in 2012.
\end{abstract}

\section{Introduction}

Throughout the paper, we only consider simple, finite, and undirected graphs.
For a graph $G$, let $|G|$ denote the order of $G$, let $\delta(G)$ 
denote the minimum degree of $G$, and let $\alpha(G)$ denote the independence number of $G$. 
For a positive integer $n$, let $K_n$ denote the complete graph of order $n$.

A graph that does not contain a \emph{claw} $K_{1,3}$ as an induced subgraph is said to be \emph{claw-free}.
A \emph{Hamilton cycle} of a graph is a cycle passing through all the vertices.
A graph with a Hamilton cycle is said to be \emph{Hamiltonian}.
Motivated by a well-known conjecture by Matthews and Sumner~\cite{MS1984} which states that every 4-connected claw-free graph is Hamiltonian, sufficient conditions for a claw-free graph to have a Hamilton cycle and related structures have been intensively studied for a long time.

A \emph{2-factor} is a 2-regular spanning subgraph of a graph.
Obviously, a Hamilton cycle is a connected 2-factor.
Choudum et al.~\cite{CP1991} and Egawa and Ota~\cite{EO1991} independently showed that every claw-free graph $G$ with $\delta(G)\geq 4$ has a 2-factor by using a criterion for the existence of a 2-factor by Tutte~\cite{T1952}.
Since a 2-factor with fewer cycles is closer to a Hamilton cycle, sufficient conditions for a claw-free graph to have a 2-factor with a bounded number of cycles have been studied.
Faudree et al.~\cite{FFFLL1999} showed that a claw-free graph $G$ with $\delta(G)\geq 4$ has a 2-factor with at most $\frac{6|G|}{\delta(G)+1}-1$ cycles, and Broersma et al.~\cite{BPY2009} showed an almost tight bound $\max\left\{\frac{n-3}{\delta(G)-1},1\right\}$ of the number of cycles with the same assumption.

For a graph $G$ and a positive integer $k$, let 
$$\sigma_k(G):=\min\left\{\sum_{v\in I}d_G(v)\;\middle|\; I\text{ is an independent set of order }k\text{ of }G\right\}$$
if $\alpha(G)\geq k$ and let $\sigma_k(G)=\infty$ otherwise.
Fron\v{c}ek et al.~\cite{FRS2004} showed the following minimum degree sum condition for a claw-free graph to have a 2-factor with a bounded number of cycles.

\begin{theorem}[\cite{FRS2004}]\label{thm:sigma k}
    Let $k\geq 3$ be an integer and let $G$ be a claw-free graph of order $n\geq 3(k+1)^2-3$.
    If $\delta(G)\geq 3k-1$ and $\sigma_{k+1}(G)>n+k^2-2k+4$, then $G$ has a 2-factor with at most $k$ cycles.
\end{theorem}

In particular, setting $k=\alpha(G)$ for a claw-free graph $G$, Theorem~\ref{thm:sigma k} implies that if $|G|\geq 3(\alpha(G)+1)^2-3$ and $\delta(G)\geq 3\alpha(G)-1$, then $G$ has a 2-factor with at most $\alpha(G)$ cycles.
Faudree et al.~\cite{FMOY2012} showed the following degree conditions for a claw-free graph to have a 2-factor with exactly $\alpha(G)$ cycles.

\begin{theorem}[\cite{FMOY2012}]\label{thm:independent}
    Let $G$ be a claw-free graph of order $n$ such that $\delta(G)\geq \frac{2n}{\alpha(G)}-2$ and $n\geq \frac{3\alpha(G)^3}{2}$.
    For every maximum independent set $S$ of $G$, $G$ has a 2-factor with $\alpha(G)$ cycles such that each cycle contains exactly one vertex of $S$.
\end{theorem}

Note that every claw-free graph $G$ of order $n$ satisfies $\delta(G)\geq \frac{2n}{\alpha(G)}-2$.
Combining two assumptions $\delta(G)\geq \frac{2n}{\alpha(G)}-2$ and $n\geq \frac{3\alpha(G)^3}{2}$, we obtain a minimum degree condition of $\delta(G)\geq 3\alpha(G)^2-2$.
They conjectured in the same paper that the lower bound of the minimum degree can be much smaller as follows.

\begin{conjecture}[\cite{FMOY2012}]\label{conj:independent}
    Let $G$ be a claw-free graph with $\delta(G)\geq \alpha(G)+1$.
    Then $G$ has a 2-factor with exactly $\alpha(G)$ cycles.
\end{conjecture}

They verified that the degree condition $\delta(G)\geq \alpha(G)+1$ can not be improved in terms of the existence of a 2-factor with exactly $\alpha(G)$ cycles.
For 2-connected claw-free graphs, Ku\v{z}el et al.~\cite{KOY2012} showed the following.

\begin{theorem}[\cite{KOY2012}]\label{thm:2conn independent}
    Let $G$ be a 2-connected claw-free graph.
    For every maximum independent set $S$ of $G$, $G$ has a 2-factor with at most $\alpha(G)$ cycles such that each cycle contains at least one vertex of $S$.
\end{theorem}

As a corollary of Theorem~\ref{thm:2conn independent}, it follows that every 2-connected claw-free graph $G$ has a 2-factor with at most $\alpha(G)$ cycles.

\subsection{Main result}\label{subsec:main result}

For a graph $G$ (not necessarily claw-free) and an independent set $I$ of $G$, let $\delta_G(I)$ denote the minimum degree of vertices contained in $I$.
The author showed the following in \cite{Karxiv}, which is a new type degree condition for a graph to have a 2-factor.

\begin{theorem}[\cite{Karxiv}]\label{thm:general graph}
    Let $G$ be a graph.
    If every independent set $I$ of $G$ satisfies $|I|\leq \delta_G(I)-1$, then $G$ has a 2-factor.
\end{theorem}

In the same paper, the author conjectured the following on the existence of a 2-factor with a bounded number of cycles.

\begin{conjecture}[\cite{Karxiv}]\label{conj:general graph}
    Let $k$ be a positive integer and let $G$ be a graph of order $n$.
    If $\sigma_{k+1}(G)\geq n$ and every independent set $I$ of $G$ satisfies $|I|\leq \delta_G(I)-1$, then $G$ has a 2-factor with at most $k$ cycles.
\end{conjecture}

In this paper, we show the following, which implies that Conjecture~\ref{conj:general graph} holds for claw-free graphs.

\begin{theorem}\label{thm:main}
    Let $k$ be a positive integer and let $G$ be a claw-free graph of order $n$.
    If $\sigma_{k+1}(G)\geq n$ and every independent set $I$ of $G$ satisfies $|I|\leq \delta_G(I)-1$, then $G$ has a 2-factor with at most $k$ cycles.
\end{theorem}

The minimum degree sum condition $\sigma_{k+1}(G)\geq n$ is best possible for any $k\geq 1$.
Indeed, for a positive integer $k$, let $H_0$ be a complete graph of order $k+3$ with vertices $\{v_1,v_2,\dots , v_{k+3}\}$ and let $H_1, H_2, \dots , H_k$ be $k$ disjoint copies of the complete graph $K_{k+2}$.
The graph $G_k$ is obtained from $H_0, H_1, \dots , H_k$ by joining all the vertices of $H_i$ and the vertex $v_i\in V(H_0)$ for each $i\in \{1,2,\dots ,k\}$.
By the definition, $G_k$ is a claw-free graph of order $|G_k|=k+3+k(k+2)=k^2+3k+3$.
Since $\delta(G_k)=k+2$ and $\alpha(G_k)=k+1$, every independent set $I$ of $G_k$ satisfies $|I|\leq \alpha(G_k)=\delta(G_k)-1\leq \delta_{G_k}(I)-1$.
Let $u_i$ be a vertex of $H_i$ for each $i\in \{1,2,\dots ,k\}$.
Considering an independent set $I=\{v_{k+1},u_1,u_2,\dots ,u_k\}$ of $G_k$, we infer that $\sigma_{k+1}(G_k)=(k+1)\delta(G_k)=(k+1)(k+2)=n-1$.
On the other hand, since each of $\{v_1, v_2, \dots , v_k\}$ is a cut vertex of $G_k$ and $G-\{v_1,v_2,\dots , v_k\}$ has $k+1$ components, every 2-factor of $G_k$ consists of at least $k+1$ cycles.

If a graph $G$ satisfies $\delta(G)\geq \alpha(G)+1$, then every independent set $I$ of $G$ satisfies $|I|\leq \alpha(G)\leq \delta(G)-1\leq \delta_G(I)-1$.
Thus, the following corollary directly follows from Theorem~\ref{thm:main}.

\begin{corollary}\label{cor:ind-mindeg}
    Let $k$ be a positive integer and let $G$ be a claw-free graph of order $n$. 
    If $\sigma_{k+1}(G)\geq n$ and $\delta(G)\geq \alpha(G)+1$, then $G$ has a 2-factor with at most $k$ cycles.
\end{corollary}

In particular, when $k$ and $G$ satisfy $k\geq \frac{\alpha(G)+2}{3}$, 
the assumption $\delta(G)\geq 3k-2$ implies $\delta(G)\geq \alpha(G)+1$.
Thus, for such $k$ and $G$, Theorem~\ref{thm:sigma k} follows from Corollary~\ref{cor:ind-mindeg}.

By setting $k=\alpha(G)$, we have the following corollary, which is a partial result of Conjecture~\ref{conj:independent}.

\begin{corollary}\label{cor:at most independent}
    Let $G$ be a claw-free graph with $\delta(G)\geq \alpha(G)+1$.
    Then $G$ has a 2-factor with at most $\alpha(G)$ cycles.
\end{corollary}

The paper is organized as follows; we introduce the closure concept of claw-free graphs by Ryj\'{a}\v{c}ek~\cite{R1997} in Section~\ref{sec:closure}, and we prove Theorem~\ref{thm:preimage} in Section~\ref{sec:proof}, which is equivalent to our main result.

\section{The closure concept of claw-free graphs}\label{sec:closure}

In 1997, Ryj\'{a}\v{c}ek~\cite{R1997} introduced the concept of closure of claw-free graphs.
For a claw-free graph $G$, the \emph{closure} of $G$, denoted by $\text{cl}(G)$, is the graph obtained from $G$ by repetition of the completion of the neighbors of a locally connected vertex as long as possible.
It was shown by Ryj\'{a}\v{c}ek~\cite{R1997} that the closure is well-defined and remains a claw-free graph.
Furthermore, the following basic properties were shown by Ryj\'{a}\v{c}ek~\cite{R1997}.
For a graph $H$, let $L(H)$ denote the line graph of $H$.

\begin{theorem}[\cite{R1997}]\label{thm:closure hamilton}
    Let $G$ be a claw-free graph and let $\text{cl}(G)$ be the closure of $G$.
    Then followings are true.
    \begin{enumerate}
        \item There is a triangle-free graph $H$ such that $\text{cl}(G)=L(H)$.
        \item If $\text{cl}(G)$ has a Hamilton cycle, then $G$ has a Hamilton cycle.
    \end{enumerate}
\end{theorem}

Analogously, on the existence of a 2-factor, Ryj\'{a}\v{c}ek et al.~\cite{RSS1999} showed the following.

\begin{theorem}[\cite{RSS1999}]\label{thm:closure 2factor}
    Let $G$ be a claw-free graph and let $\text{cl}(G)$ be the closure of $G$.
    If $\text{cl}(G)$ has a 2-factor with $k$ cycles, then $G$ has a 2-factor with at most $k$ cycles.
\end{theorem}

Therefore, in order to find a 2-factor with few cycles in a claw-free graph $G$, it suffices to find a 2-factor of the closure $\text{cl}(G)$ with few cycles.

For a graph $H$, a closed trail $C$ of $H$ is said to be \emph{dominating} if every edge $e\in E(H)\setminus E(C)$ has at least one end that is contained in $C$.
Harary and Nash-Williams~\cite{HN1965} showed the following, which states that the existence of a dominating closed trail of $H$ is equivalent to the existence of a Hamilton cycle of its line graph.

\begin{theorem}[\cite{HN1965}]\label{thm:harary NW}
    Let $H$ be a graph without isolated vertices.
    The line graph $L(H)$ has a Hamilton cycle if and only if either $H$ has a dominating closed trail or $H$ is a star with at least three edges.
\end{theorem}

For a graph $H$, a \emph{dominating system of $H$} is a set $\mathcal{D}=\{D_1,D_2,\dots , D_k\}$ of subgraphs of $H$ such that 
\begin{itemize}
    \item $D_1,D_2,\dots ,D_k$ are pairwise edge-disjoint,
    \item for each $i\in \{1,2,\dots ,k\}$, $D_i$ is either a closed trail or a star with at least three edges, and 
    \item for each $e\in E(H)\setminus\left(\bigcup_{i=1}^k E(D_i)\right)$, at least one end of $e$ is contained in a closed trail in $\mathcal{D}$.
\end{itemize}
The number of subgraphs in a dominating system $\mathcal{D}$ is called the \emph{cardinality} of $\mathcal{D}$.
Gould and Hynds~\cite{GH1999} showed the following as a generalization of Theorem~\ref{thm:harary NW}.

\begin{theorem}[\cite{GH1999}]\label{thm:k system}
    Let $H$ be a graph without isolated vertices.
    The line graph $L(H)$ has a 2-factor with $k$ cycles if and only if $H$ has a dominating system of cardinality $k$.
\end{theorem}

By Theorems~\ref{thm:closure hamilton}, \ref{thm:closure 2factor} and \ref{thm:k system}, in order to show our main result Theorem~\ref{thm:main}, it suffices to show the following.

\begin{theorem}\label{thm:preimage}
    Let $k$ be a positive integer.
    Let $H$ be a triangle-free graph and let $G$ be the line graph of $H$.
    If $\sigma_{k+1}(G)\geq |G|$ and every independent set $I$ of $G$ satisfies $|I|\leq \delta_G(I)-1$, then $H$ has a dominating system of cardinality at most $k$.
\end{theorem}

\begin{proof}[Proof of Theorem~\ref{thm:main}]
    Let $k$ be a positive integer and let $G$ be a claw-free graph of order $n$ such that $\sigma_{k+1}(G)\geq |G|$ and every independent set $I$ of $G$ satisfies $|I|\leq \delta_G(I)-1$.
    Since the closure $\text{cl}(G)$ is obtained from $G$ by adding edges, $\text{cl}(G)$ satisfies $\sigma_{k+1}(\text{cl}(G))\geq n=|\text{cl}(G)|$ and every independent set $I'$ of $\text{cl}(G)$ satisfies $|I'|\leq \delta_G(I')-1\leq \delta_{\text{cl}(G)}(I')-1$.
    By Theorem~\ref{thm:closure hamilton}, there is a triangle-free graph $H$ such that $\text{cl}(G)=L(H)$.
    Applying Theorem~\ref{thm:preimage} to $H$, we obtain a dominating system of $H$ of cardinality at most $k$.
    Thus, by Theorems~\ref{thm:k system} and \ref{thm:closure 2factor}, $\text{cl}(G)=L(H)$ has a 2-factor with at most $k$ cycles and so is $G$.
\end{proof}

\section{Proof of Theorem~\ref{thm:preimage}}\label{sec:proof}

We fix a positive integer $k$.
Let $H$ be a triangle-free graph such that the line graph $G$ of $H$ satisfies $\sigma_{k+1}(G)=|G|\geq n$ and every independent set $I$ of $G$ satisfies $|I|\leq \delta_G(I)-1$.
Suppose that $H$ does not have a dominating system of cardinality at most $k$.

Since every independent set $I$ of $G$ satisfies $|I|\leq \delta_G(I)-1$, by Theorem~\ref{thm:general graph}, $G$ has a 2-factor, and hence $H$ has a dominating system.
Let $\mathcal{D}=\{D_1,D_2,\dots, D_\ell\}$ be a dominating system of $H$ so that 
\begin{enumerate}[label=(\arabic*)]
    \item the cardinality $\ell$ is as small as possible, and 
    \item subject to (1), $|\bigcup_{i=1}^\ell E(D_i)|$ is as large as possible.
\end{enumerate}
By our assumption on $H$, we have $\ell\geq k+1$.

Let $\mathcal{D}_1$ be the set of closed trails in $\mathcal{D}$, and let $\mathcal{D}_2$ be the set of stars in $\mathcal{D}$.
For $i=1,2$, let $V(\mathcal{D}_i)$ and $E(\mathcal{D}_i)$ respectively denote $\bigcup_{D\in \mathcal{D}_i}V(D)$ and $\bigcup_{D\in \mathcal{D}_i}E(D)$.
For each $D\in\mathcal{D}_2$, let $x_D$ be the center of $D$, and let $X=\{x_D\mid D\in\mathcal{D}_2\}$.
The following three can be easily shown from the minimality of $\ell$.
\begin{enumerate}[label=(\alph*)]
    \item\label{cond:ct ct} For any distinct closed trails $D$ and $D'$ in $\mathcal{D}_1$, $V(D)\cap V(D')=\emptyset$.
    \item\label{cond:star star} For any distinct stars $D$ and $D'$ in $\mathcal{D}_2$, $x_D\neq x_{D'}$.
    \item\label{cond:star ct} $V(\mathcal{D}_1)\cap X=\emptyset$.
\end{enumerate}
Furthermore, the following claim holds from the choice of $\mathcal{D}$.

\begin{claim}\label{claim:cycle}
    For any cycle $C$ of $H$, there is a closed trail $D\in \mathcal{D}_1$ such that $|E(C)\cap E(D)|\geq 2$.
    In particular, $|E(C)\cap E(\mathcal{D}_1)|\geq 2$ for any cycle $C$ of $H$.
\end{claim}

\begin{proof}
    Assume to the contrary that $H$ has a cycle $C$ such that $|E(C)\cap E(D)|\leq 1$ for every closed trail $D\in\mathcal{D}_1$.
    Let $\mathcal{D}_C=\{D\in\mathcal{D}_1\mid V(C)\cap V(D)\neq \emptyset\}\cup \{D\in\mathcal{D}_2\mid x_D\in V(C)\}$.
    Since $\mathcal{D}$ is a dominating system, we have $|\mathcal{D}_C|\geq 1$.
    Let $C'$ be a subgraph of $H$ induced by the symmetric difference of $E(C)$ and $\bigcup_{D\in \mathcal{D}_C\cap \mathcal{D}_1}E(D)$.
    Since $|E(C)\cap E(D)|\leq 1$ for each $D\in \mathcal{D}_C\cap \mathcal{D}_1$ and closed trails in $\mathcal{D}_C\cap \mathcal{D}_1$ are pairwise vertex-disjoint, $C'$ is connected and Eulerian.
    Hence, $C'$ is a closed trail of $H$ that is edge-disjoint from subgraphs in $\mathcal{D}\setminus \mathcal{D}_C$.
    Let $\mathcal{D}'=(\mathcal{D}\setminus \mathcal{D}_C)\cup \{C'\}$.
    For each $D\in\mathcal{D}_C\cap \mathcal{D}_1$, we have $V(D)\subseteq V(C')$ and hence all the edges contained or dominated by $D$ are contained or dominated by $C'$.
    For each $D\in \mathcal{D}_C\cap \mathcal{D}_2$, since $x_D\in V(C)\subseteq V(C')$, all the edges of $D$ are dominated by $C'$.
    Thus, $\mathcal{D}'$ is a dominating system of $H$.
    If $|\mathcal{D}_C|\geq 2$, then $\mathcal{D}'$ has the cardinality less than $\ell$, a contradiction by the minimality of $\ell$.
    Otherwise, there is a closed trail $D^*\in\mathcal{D}_1$ such that $\mathcal{D}_C=\{D^*\}$.
    Since $|E(C)\cap E(D^*)|\leq 1$ and $|E(C)\cap E(D)|=0$ for every $D\in \mathcal{D}\setminus\{D^*\}$, it follows that
    \begin{align*}
        \Bigl|\bigcup_{D\in \mathcal{D}'}E(D)\Bigr|=\Bigl|\bigcup_{D\in\mathcal{D}}E(D)\Bigr|-|E(C)\cap E(D^*)|+|E(C)\setminus E(D^*)|>\Bigl|\bigcup_{D\in\mathcal{D}}E(D)\Bigr|,
    \end{align*}
    a contradiction by the maximality of $|\bigcup_{i=1}^\ell E(D_i)|$.
\end{proof}
    
For each $D\in\mathcal{D}_2$, let $T(D)$ be the set of leaves of $D$ that are not contained in $V(\mathcal{D}_1)\cup X$.
Let $T=\bigcup_{D\in \mathcal{D}_2}T(D)$, and let $F$ be the subgraph of $H$ such that $V(F)=X\cup T$ and $E(F)=\bigcup_{D\in \mathcal{D}_2} \{x_Dt\mid t\in T(D)\}$.
Note that $F$ is a bipartite graph with partite sets $X$ and $T$.
Since $\mathcal{D}$ dominates $H$, $T$ is an independent set of $H$.
The proof is divided into two cases depending on whether $F$ contains a matching that covers $X$ or not.

\vspace{\baselineskip}
\noindent
\textit{Case 1:} $F$ contains a matching that covers $X$.

Let $M$ be a matching of $F$ that covers $X$.
By choosing arbitrary one edge from each closed trail in $\mathcal{D}_1$, we extend $M$ to $M'\subseteq E(H)$ such that $|M'\cap E(D)|=1$ for every $D\in \mathcal{D}$.
By \ref{cond:ct ct}, \ref{cond:star star} and \ref{cond:star ct}, $M'$ is a matching of $H$, and hence $M'$ is an independent set of $G$ of order $\ell\geq k+1$.
Let $F'$ be the subgraph of $H$ induced by all the end vertices of $M'$.
Since each edge in $E(H)\setminus E(F')$ has at most one neighbor in $M'$ in $G$ and each edge in $E(F')\setminus M'$ has two neighbors in $M'$ in $G$, the assumption of $\sigma_{k+1}(G)\geq n$ implies that
\begin{align*}
    n\leq \sum_{e\in M'}d_G(e)
    \leq |E(H)\setminus E(F')|+2|E(F')\setminus M'|
    =|E(H)|-|E(F')|-2|M'|
    =n+|E(F')|-2\ell,
\end{align*}
and thus $|E(F')|\geq 2\ell$.
Since $|F'|=2|M'|=2\ell$, $F'$ contains a cycle $C$.
By \ref{cond:ct ct}, \ref{cond:star ct}, and the fact $V(\mathcal{D}_1)\cap T=\emptyset$, no edge of $E(F')\setminus M'$ belongs to $E(\mathcal{D}_1)$.
This implies that $|E(C)\cap E(D)|\leq |M'\cap E(D)|=1$ for every $D\in\mathcal{D}_1$, a contradiction by Claim~\ref{claim:cycle}.

\vspace{\baselineskip}
\noindent
\textit{Case 2:} $F$ does not contain a matching that covers $X$.

By Hall's Theorem, there is a subset $S\subseteq X$ such that $|N_F(S)|<|S|$.
Choose the smallest $S$.
Let $\mathcal{D}_S=\{D\in\mathcal{D}_2\mid x_D\in S\}$.

\begin{claim}\label{claim:S independent}
    $S$ is an independent set of $H$.
\end{claim}

\begin{proof}
    Suppose that there are two stars $D$ and $D'$ in $\mathcal{D}_S$ such that $x_{D}x_{D'}\in E(H)$.
    Note that $x_{D}x_{D'}\in E(D)\cup E(D')$ by \ref{cond:star star} and \ref{cond:star ct}. 
    Let $S_1$ be a set of vertices $x\in S$ such that there is an $x_Dx$-path of $F$, and let $S_2=S\setminus S_1$.
    If there is an $x_{D}x_{D'}$-path $P$ of $F$, then $P+x_{D}x_{D'}$ is a cycle of $H$ satisfying $E(P+x_{D}x_{D'})\subseteq E(\mathcal{D}_2)$, which contradicts Claim~\ref{claim:cycle}.
    Thus, $x_{D'}\in S_2$ and hence both $S_1$ and $S_2$ are nonempty.
    Since $N_F(S_1)\cap N_F(S_2)=\emptyset$ and $N_F(S_1)\cup N_F(S_2)=N_F(S)$, we have
    \begin{align*}
        |N_F(S_1)|+|N_F(S_2)|=|N_F(S)|<|S|=|S_1|+|S_2|.
    \end{align*}
    Therefore, either $|N_F(S_1)|<|S_1|$ or $|N_F(S_2)|<|S_2|$ is satisfied, a contradiction to the minimality of the order of $S$.
\end{proof}

In the rest of the proof, we consider two cases depending on whether every vertex $t\in N_F(S)$ satisfies $|N_F(t)\cap S|\geq 3$ or not.

\vspace{\baselineskip}
\noindent
\textit{Case 2-1}: Every vertex $t\in N_F(S)$ satisfies $|N_F(t)\cap S|\geq 3$.

For each $t\in N_F(S)$, let $D_t$ be a star induced by the edges of $F$ joining $t$ and $N_F(t)\cap S$.
By the assumption of this case, each $D_t$ is a star with at least three edges.
For each $D\in \mathcal{D}_2\setminus \mathcal{D}_S$, let $\overline{D}$ be a star obtained from $D$ by adding all the edges that join $x_D$ and $S$.
Let $\mathcal{D}'=\mathcal{D}_1\cup \{\overline{D}\mid D\in  \mathcal{D}_2\setminus \mathcal{D}_S\}\cup \{D_t\mid t\in N_F(S)\}$.

By Claim~\ref{claim:S independent}, each edge of $\bigcup_{D\in \mathcal{D}_S}E(D)$ is either contained in $F$, incident to a closed trail in $\mathcal{D}_1$, or incident to the center of a star in $\mathcal{D}_2\setminus \mathcal{D}_S$.
If $e\in \bigcup_{D\in \mathcal{D}_S}E(D)$ is an edge of $F$, then $e\in E(D_t)$ for some $t\in N_F(S)$.
If $e\in \bigcup_{D\in \mathcal{D}_S}E(D)$ is incident to a closed trail $D'\in\mathcal{D}_1$, then $e$ is dominated by $D'$.
If $e\in \bigcup_{D\in\mathcal{D}_S}E(D)$ is incident to the center of a star $D'\in\mathcal{D}_2\setminus\mathcal{D}_S$, then $e\in E(\overline{D'})$ by the definition of $\overline{D'}$.
Combining these, we conclude that each edge of $\bigcup_{D\in \mathcal{D}_S}E(D)$ is either contained in a star in $\mathcal{D}'$ or dominated by a closed trail in $\mathcal{D}'$, and hence $\mathcal{D}'$ is a dominating system of $H$.
    
On the other hand, since $|N_F(S)|<|S|$, we have $|\mathcal{D}'|= |\mathcal{D}|-|S|+|N_F(S)|<|\mathcal{D}|$, a contradiction to the minimality of $\ell$.

\vspace{\baselineskip}
\noindent
\textit{Case 2-2:} There is a vertex $t\in N_F(S)$ such that $|N_F(t)\cap S|\leq 2$.

Let $D$ be a star in $\mathcal{D}_S$ such that $x:=x_D\in N_F(t)$.
Let $U:=\{u_1,u_2,\dots ,u_p\}=N_H(x)\setminus\{t\}$ and let $W:=\{w_1,w_2,\dots ,w_q\}=N_H(t)\setminus S$.
Since $H$ is triangle-free, $U\cap W=\emptyset$.

\begin{claim}\label{claim:u degree}
    For every $u_i\in U$, $d_H(u_i)\geq 2$.
\end{claim}

\begin{proof}
    Assume to the contrary that $d_H(u_i)=1$ for some $u_i\in U$.
    Then it follows that $u_i\in T(D)\subseteq T$ and $N_F(u_i)=\{x\}$, and hence $|N_F(S\setminus\{x\})|\leq |N_F(S)\setminus\{u_i\}|=|N_F(S)|-1<|S|-1=|S\setminus\{x\}|$, a contradiction by the minimality of the order of $S$.
\end{proof}

Since $t\notin V(\mathcal{D}_1)\cup X$, for every $j\in \{1,2,\dots , q\}$, the edge $tw_j$ is either contained in a star of $\mathcal{D}_2$ with the center $w_j$ or dominated by a closed trail of $\mathcal{D}_1$, which implies that $d_H(w_j)\geq 2$.

For each closed trail $D'\in \mathcal{D}_1$, we fix an orientation of $D'$ and write $\overrightarrow{D'}$ with its orientation.
We define $\varphi(v)\in V(H)$ for each $v\in U\cup W$ as follows.
\begin{itemize}
    \item If $v$ is a vertex of a closed trail $\overrightarrow{D'}$ of $\mathcal{D}_1$, then let $\varphi(v)$ be the vertex that follows $v$ along $\overrightarrow{D'}$.
    \item Otherwise, since $d_H(v)\geq 2$, let $\varphi(v)$ be an arbitrary vertex in $N_H(v)$ distinct from $x$ or $t$. 
\end{itemize}
We use the notation $\varphi(U):=\{\varphi(u_i)\mid 1\leq i\leq p\}$ and $\varphi(W):=\{\varphi(w_j)\mid 1\leq j\leq q\}$.
Since $H$ is triangle-free, it follows that $\{t,x\}\cap (\varphi(U)\cup \varphi(W))=\emptyset$, $U\cap \varphi(U)=\emptyset$, and $W\cap \varphi(W)=\emptyset$.

\begin{claim}\label{claim:edge between uw}
    There is no edge between $U$ and $W$.
    In particular, $U\cap \varphi(W)=W\cap \varphi(U)=\emptyset$.
\end{claim}

\begin{proof}
    Assume to the contrary that there is an edge between $U$ and $W$.
    Without loss of generality, we may assume that $u_1w_1\in E(H)$.
    Then $C=txu_1w_1t$ is a 4-cycle of $H$.
    Since $x, t\notin V(\mathcal{D}_1)$ by the definition, it follows that $xu_1, xt, tw_1\notin E(\mathcal{D}_1)$, and hence $|E(C)\cap E(\mathcal{D}_1)|\leq 1$, a contradiction by Claim~\ref{claim:cycle}.
\end{proof}

Therefore, $\varphi$ is a mapping from $U\cup W$ to $V(H)\setminus (U\cup W)$.

\begin{claim}\label{claim:phi inj}
    $\varphi$ is an injection from $U\cup W$ to $V(H)\setminus (U\cup W)$.
\end{claim}

\begin{proof}
    Assume to the contrary that there are two distinct vertices $v_1, v_2\in U\cup W$ such that $\varphi(v_1)=\varphi(v_2)=y$ for some $y\in V(H)\setminus (U\cup W)$.
    Let $P$ be a $v_1v_2$-path of $H$ such that $V(P)\subseteq \{v_1,v_2,t,x\}$, and let $C=v_1Pv_2yv_1$. 
    Since $C$ is a cycle of $H$, by Claim~\ref{claim:cycle}, at least two edges of $C$ belong to $E(\mathcal{D}_1)$.
    This together with the fact that $x,t\notin V(\mathcal{D}_1)$ implies that $v_1y,v_2y\in E(\mathcal{D}_1)$.
    By the property \ref{cond:ct ct}, there is a closed trail $D'\in \mathcal{D}_1$ such that $v_1y,v_2y\in E(D')$.

    Let $C'$ be a subgraph of $H$ induced by the symmetric difference of $E(D')$ and $E(C)$.
    By the choice of $\varphi(v_1)$ and $\varphi(v_2)$, there are $yv_1$-walk $P_1$ and $yv_2$-walk $P_2$ of $D'-\{v_1y,v_2y\}$, which implies that $C'$ is connected.
    Thus, $C'$ is a closed trail of $H$.
    Furthermore, the properties \ref{cond:ct ct} and \ref{cond:star ct} imply that $C'$ is edge-disjoint from every subgraph in $\mathcal{D}\setminus \{D,D'\}$.
    Then $(\mathcal{D}\setminus\{D,D'\})\cup \{C'\}$ is a dominating system of $H$ of cardinality less than $\ell$, a contradiction.
\end{proof}

Let $N=\{xt\}\cup \{v\varphi(v)\mid v\in U\cup W\}$.
By Claims~\ref{claim:edge between uw} and \ref{claim:phi inj}, $N$ is a matching of $H$ with $p+q+1$ edges, and equivalently $N$ is an independent set of $G$ of order $p+q+1$.
On the other hand, since 
\begin{align*}
    d_G(xt)=d_H(x)+d_H(t)-2\leq (p+1)+(q+2)-2=p+q+1,
\end{align*}
we infer that 
\begin{align*}
    |N|=p+q+1\geq d_G(xt)\geq \delta_G(N),
\end{align*}
a contradiction by the assumption that every independent set $I$ of $G$ satisfies $|I|\leq \delta_G(I)-1$.
This completes the proof of Theorem~\ref{thm:preimage}.

\section{Concluding remark}

In our proof of Theorem~\ref{thm:preimage}, the assumption that every independent set $I$ of $G$ satisfies $|I|\leq \delta_G(I)-1$ is used only for the existence of a 2-factor and Case 2.2.
By considering a dominating system of $H$ that allows stars with one or two edges, our proof of Theorem~\ref{thm:preimage} shows the following, which is a result on a degenerate cycle partition of claw-free graphs.

\begin{theorem}\label{thm:degenerate cycle}
    Let $k$ be a positive integer and let $G$ be a claw-free graph of order $n$.
    If $\sigma_{k+1}(G)\geq n$, then there is a partition of $V(G)$ into at most $k$ parts such that each part induces $K_1$, $K_2$ or a Hamiltonian subgraph of $G$.
\end{theorem}

\section*{Acknowledgement}

The author thanks Professor Katsuhiro Ota for giving valuable comments on this paper.
The author has been partially supported by Keio University SPRING scholarship JPMJSP2123 and JST ERATO JPMJER2301.


\begin{thebibliography}{99}
    \bibitem{BPY2009}
    H.~Broersma, D.~Paulusma, and K.~Yoshimoto, Sharp upper bounds on the number of components of 2-factors in claw-free graphs, Graphs Combin. \textbf{25} (2009), 427--460.
    \bibitem{CP1991}
    S.~A.~Choudum and M.~S.~Paulraj, Regular factors in $K_{1,3}$-free graphs, J. Graph Theory, \textbf{15} (1991), 259--265.
    \bibitem{EO1991}
    Y.~Egawa and K.~Ota, Regular factors in $K_{1,n}$-free graphs,
    J. Graph Theory \textbf{15} (1991), 337--344.
    \bibitem{FFFLL1999}
    R.~J.~Faudree, O.~Favaron, E.~Flandrin, H.~Li, and Z.~Liu, On 2-factors in claw-free graphs, Discrete Math. \textbf{206} (1999), 131--137.
    \bibitem{FMOY2012}
    R.~J.~Faudree, C.~Magnant, K.~Ozeki, and K.~Yoshimoto, Claw-free graphs and 2-factors that separate independent vertices, J. Graph Theory \textbf{69} (2012), 251--263.
    \bibitem{FRS2004}
    D.~Fon\v{c}ek, Z.~Ryj\'{a}\v{c}ek, and Z.~Skupie\'{n}, On traceability and 2-factors in claw-free graphs, Discuss. Math. Graph Theory \textbf{24} (2004), 55--71.
    \bibitem{GH1999}
    R.~J.~Gould and E.~A.~Hynds, A note on cycles in 2-factors of line graphs, Bull. ICA \textbf{26} (1999), 46--48.
    \bibitem{Karxiv}
    M.~Kashima, New type degree conditions for a graph to have a 2-factor, arXiv:2503.18409.
    \bibitem{KOY2012}
    R.~Ku\v{z}el, K.~Ozeki, and K.~Yoshimoto, 2-factors and independent sets on claw-free graphs, Discrete Math. \textbf{312} (2012), 202--206.
    \bibitem{HN1965}
    F.~Harary and C.~St.J.A.~Nash-Williams, On eulerian and hamiltonian graphs and line graphs, Canad. Math. Bull. \textbf{8} (1965), 701--709.
    \bibitem{MS1984}
    M.~M.~Matthews and D.~P.~Sumner, Hamiltonian results in $K_{1,3}$-free graphs, J. Graph Theory \textbf{8} (1984), 139--146.
    \bibitem{R1997}
    Z.~Ryj\'{a}\v{c}ek, On a closure concept in claw-free graphs, J. Combin. Theory Ser. B \textbf{70} (1997), 217--224.
    \bibitem{RSS1999}
    Z.~Ryj\'{a}\v{c}ek, A.~Saito, and R.~H.~Schelp, Closure, 2-factors, and cycle coverings in claw-free graphs, J. Graph Theory \textbf{32} (1999), 109--117.
    \bibitem{T1952}
    W.~T.~Tutte, The factors of graphs, Canadian J. Math. \textbf{4} (1952), 314--328.
\end{thebibliography}
\end{document}